\documentclass{amsart}
\usepackage{amssymb,amscd,verbatim,amsthm,amsmath,amsfonts,amsgen,xspace}
\usepackage{graphicx, color}

\newtheorem{theorem}{Theorem}[section]
\newtheorem{definition}[theorem]{Definition}

\newtheorem{corollary}[theorem]{Corollary}

\newtheorem{example}[theorem]{Example}

\newtheorem{proposition}[theorem]{Proposition}
\newtheorem{remark}[theorem]{Remark}
\newtheorem{question}[theorem]{Question}

\newcommand{\mN}{\mathbb N}

\newcommand{\mC}{\mathbb C}

\newcommand{\mR}{\mathbb R}

\newcommand{\be}{\begin{eqnarray}}
\newcommand{\ee}{\end{eqnarray}}
\newcommand{\bd}{\begin{definition}}
\newcommand{\ed}{\end{definition}}

\newcommand{\br}{\begin{remark}}
\newcommand{\er}{\end{remark}}
\newcommand{\gog}{{\mathfrak g}}

\newcommand{\gos}{{\mathfrak s}}
\newcommand{\gop}{{\mathfrak p}}

\newcommand{\gon}{{\mathfrak n}}
\newcommand{\gol}{{\mathfrak l}}

\newcommand{\gor}{{\mathfrak r}}
\newcommand{\bt}{\begin{tabular}}
\newcommand{\et}{\end{tabular}}

\def\ker{\operatorname{Ker}}
\def\im{\operatorname{Im}}

\def\Coker{\operatorname{Coker}}
\def\Ker{\operatorname{Ker}}
\def\Im{\operatorname{Im}}

\def\top{\operatorname{top}}

\def\Spin{\operatorname{Spin}}
\def\Cas{\operatorname{Cas}}

\newcommand{\Kt}{\tilde{K}}

\newcommand{\pf}{\begin{proof}}
\newcommand{\epf}{\end{proof}}
\newcommand{\eq}{\begin{equation}}
\newcommand{\eeq}{\end{equation}}

\newcommand{\fra}{\mathfrak{a}}
\newcommand{\frb}{\mathfrak{b}}

\newcommand{\frg}{\mathfrak{g}}
\newcommand{\frh}{\mathfrak{h}}

\newcommand{\frk}{\mathfrak{k}}
\newcommand{\frl}{\mathfrak{l}}

\newcommand{\frn}{\mathfrak{n}}
\newcommand{\fro}{\mathfrak{o}}
\newcommand{\frp}{\mathfrak{p}}

\newcommand{\frr}{\mathfrak{r}}
\newcommand{\frs}{\mathfrak{s}}
\newcommand{\frt}{\mathfrak{t}}

\newcommand{\bbC}{\mathbb{C}}

\newcommand{\bbN}{\mathbb{N}}
\newcommand{\bbR}{\mathbb{R}}
\newcommand{\bbZ}{\mathbb{Z}}

\input xypic
\input epsf
\input xy
\xyoption{all}

\begin{document}

\title[Higher Dirac cohomology]{Higher Dirac cohomology of modules with generalized infinitesimal character}
\author{Pavle Pand\v zi\'c}
\address[P. Pand\v zi\'c]{Department of Mathematics, University of Zagreb, Bijeni\v cka 30, 10000 Zagreb, Croatia}
	\email{pandzic@math.hr}
	\thanks{The author was partially supported by a grant from Ministry of science, education and sport of Republic of Croatia.}
\author{Petr Somberg}
\address[P. Somberg]{Mathematical Institute of Charles University, Sokolovsk\'a 83, 18000 Praha 8 - Karl\'{\i}n, Czech Republic}
\email{somberg@karlin.mff.cuni.cz}
\thanks{The author was supported by grant GA \v CR P201/12/G028.}

\begin{abstract} 
We modify the definition of Dirac cohomology in such a way that the standard properties of the usual Dirac cohomology, valid for
modules with infinitesimal character, become valid also for modules with only generalized infinitesimal character. 
\end{abstract}


\maketitle


\section{Introduction}

Dirac operators were introduced into representation theory of real reductive Lie groups by Parthasarathy \cite{P1} 
with the aim to construct discrete series representations. The definition goes as follows: let $G$ be a connected
real reductive Lie group with maximal compact subgroup $K$ and let $\frg_0=\frk_0\oplus\frp_0$ be the corresponding
Cartan decomposition of the Lie algebra of $G$. Write $\frg=\frk\oplus\frp$ for the complexifications. 
Let $B$ be a nondegenerate invariant symmetric bilinear
form on $\frg$ such that $B$ is negative definite on $\frk_0$ and positive definite on $\frp_0$, and such that $\frk$ 
and $\frp$ are orthogonal to each other with respect to $B$. If $b_i$ is any basis of $\frp$, and if $d_i$ is the dual
basis with respect to $B$, then the Dirac operator associated to the pair $(\frg,\frk)$ is
\eq
\label{def Dirac op}
D=\sum_i b_i\otimes d_i\quad \in U(\frg)\otimes C(\frp),
\eeq
where $U(\frg)$ denotes the universal enveloping algebra of $\frg$ and $C(\frp)$ denotes the Clifford algebra of $\frp$
with respect to $B$. It is easy to see that $D$ is independent of the choice of $b_i$, and $K$-invariant for the adjoint
action of $K$ on $U(\frg)\otimes C(\frp)$. Moreover, Parthasarathy proved in \cite{P1} that
\eq
\label{D squared}
D^2=(\Cas_{\frk_\Delta} +\|\rho_\frk\|^2)-(\Cas_\gog\otimes 1 +\|\rho_\gog\|^2).
\eeq
Here $\Cas_\frg$ is the Casimir element of $U(\frg)$, $\Cas_{\frk_\Delta}$ is the Casimir element of a diagonal copy of
$U(\frk)$ in $U(\frg)\otimes C(\frp)$ obtained by combining the embedding $\frk\hookrightarrow\frg$ with the map 
$\frk\to\frs\fro(\frp)\cong\bigwedge^2(\frp)\hookrightarrow C(\frp)$, and $\rho_\frg$ respectively $\rho_\frk$ are the 
half sums of positive roots for $\frg$ respectively $\frk$.

Besides the construction of the discrete series representations, this Dirac operator was also useful in several classification
schemes for certain classes of unitary representations. Namely, in \cite{P2} Parthasarathy proved a necessary criterion for 
unitarity, the so called Dirac inequality. In short, if $V$ is a unitary $(\frg,K)$-module, then there is a natural inner product
on the $U(\frg)\otimes C(\frp)$-module $V\otimes S$, where $S$ is a spin module for $C(\frp)$. $D$ is self-adjoint with respect
to this inner product, hence $D^2\geq 0$. Writing this inequality explicitly on the $\tilde K$-types of $V\otimes S$ using
(\ref{D squared}) leads to the Dirac inequality.
(Here $\tilde K$ is the spin double cover of $K$, i.e., the pullback of the cover $\Spin(\frp_0)\to SO(\frp_0)$ by the action map
$K\to SO(\frp_0)$.)

In the 1990s, Vogan revisited Parthasarathy's theory with the aim of sharpening the Dirac inequality \cite{V}. He introduced the 
notion of Dirac cohomology: for a $(\frg,K)$-module $V$, consider the action of $D$ on $V\otimes S$. The Dirac cohomology of $V$ is
the $\tilde K$-module
\eq
\label{def Dir coh}
H_D(V)=\Ker D/(\Ker D\cap \Im D). 
\eeq
For admissible $V$, $H_D(V)$ is finite-dimensional, as follows readily from (\ref{D squared}). It turns out that for most $V$ the Dirac cohomology
is zero, but when it is not zero then it contains interesting information about $V$. In particular, the following result was conjectured by
Vogan and proved by Huang and Pand\v zi\'c in \cite{HP1}. Let $\frh=\frt\oplus\fra$ be a fundamental Cartan subalgebra of $\frg$ (so that $\frt$ is a Cartan subalgebra of $\frk$ and $\fra$ is the centralizer of $\frt$ in $\frp$). We view $\frt^*\subset\frh^*$, by
extending functionals on $\frt$ by zero over $\fra$. We consider infinitesimal characters as elements of $\frh^*$ via the Harish-Chandra isomorphism. 
We fix a choice of positive roots for $(\frk,\frt)$.

\begin{theorem}
\label{Vogan conjecture}
With the above notation, let $V$ be a $(\frg,K)$-module with infinitesimal character $\Lambda\in\frh^*$. Assume that $H_D(V)$ contains a $\tilde K$-type $E_\mu$ of highest weight $\mu\in\frt^*\subset\frh^*$. Then $\Lambda$ is conjugate to $\mu+\rho_\frk$ under the Weyl group $W(\frg,\frh)$.
\end{theorem} 

It is an interesting problem to classify all irreducible unitary modules with nonvanishing Dirac cohomology, and to calculate the Dirac cohomology for each such module. Some results towards the solution of this problem can be found for example in \cite{BP1}, \cite{BP2}, \cite{hkp} and \cite{HPP}.

The above setting was generalized by Kostant to the situation where $\frk$ is replaced by a reductive quadratic subalgebra of $\frg$, i.e., a reductive subalgebra $\frr$ such that $B\big|_{\frr\times\frr}$ is nondegenerate. The appropriate replacement for the above $D$ is then Kostant's cubic Dirac operator \cite{K1}, which we again denote by $D$. If the orthogonal of $\frr$ in $\frg$ is denoted by $\frs$, so that $\frg=\frr\oplus\frs$, then
$D\in U(\frg)\otimes C(\frs)$ is defined as
\eq
\label{def cubic}
D=\sum_i b_i\otimes d_i + \frac{1}{2}\sum_{i<j<k} B([b_i,b_j],b_k)d_i\wedge d_j\wedge d_k,
\eeq
where, as before, $b_i$ and $d_i$ are dual bases of $\frs$ with respect to $B$, and the wedge product is defined
using Chevalley's identification of $C(\frs)$ and $\bigwedge(\frs)$ as vector spaces.

Kostant proved that (\ref{D squared}) and Theorem \ref{Vogan conjecture} still hold in this setting. (Dirac cohomology in this more general setting is still defined by (\ref{def Dir coh}).) Moreover, in \cite{kos} he obtained
some results about Dirac cohomology of highest weight modules with respect to a parabolic subalgebra $\frp=\frl\oplus\frn$ of $\frg$. (In this case, the above $\frr$ is equal to $\frl$.) In particular, in the equal-rank case he obtained the 
result on non-triviality of Dirac cohomology for all highest weight modules and also 
determined the Dirac cohomology of finite-dimensional modules. These investigations were continued by Huang and Xiao in \cite{hx}, where
it is proved that Dirac cohomology of simple highest weight modules is up 
to a twist equal to $\frn$-cohomology. Special cases of this were also obtained in \cite{HPR}. Another place where Dirac cohomology of
highest weight modules is studied is \cite{dh}.

We are interested in applying Dirac cohomology techniques to the study of restrictions of highest weight modules, as for example the
restrictions described in \cite{koss} and \cite{kossII}. In the course of our investigations, we noticed that while the above described
notion of Dirac cohomology works fine for modules with infinitesimal character, it does not have so good properties for modules with
only generalized infinitesimal character. Typical representative examples of this situation are given in Section \ref{examples}.

The aim of this paper is to overcome such difficulties by introducing new kinds of Dirac cohomology functors, which we call 
higher Dirac cohomology. As we shall see, the higher Dirac cohomology is the same as the usual Dirac cohomology for modules
with infinitesimal character, and it is typically bigger for modules with only generalized infinitesimal character.

Before describing the kind of properties we want our Dirac cohomology to have, let us describe our setting more precisely. We will work with
a reductive Lie algebra $\frg$ over $\bbC$, with a nondegenerate invariant symmetric bilinear form $B$, and a quadratic reductive subalgebra $\frr$.
We will consider $\frg$-modules $V$ of finite length, such that the restriction of $V$ to $\frr$ equals a direct sum of finite-dimensional irreducible
$\frr$-modules with finite multiplicities. Moreover, we assume that $V$ is a direct sum of submodules with generalized infinitesimal character. (Recall that a $\frg$-module $W$ has generalized infinitesimal character $\chi$ if there is $N\in\bbN$ such that $(z-\chi(z))^N=0$ on $W$ for each $z$ in the center $Z(\frg)$ of $U(\frg)$.)

Note that any $V$ as above is finitely generated over $U(\frg)$, and also $Z(\frg)$-finite. In case when the pair $(\frg,\frr)$ comes from a real reductive Lie group $G$ and its maximal compact subgroup $K$, then some of the above assumptions are redundant, i.e., follow from the other assumptions.

In the rest of the article all $(\frg,\frr)$-modules we consider will satisfy these assumptions, and we will call such modules simply 
$(\frg,\frr)$-modules.

Let us now assume that $\frg$ and $\frr$ have equal rank. Let $\frh$ be a common Cartan subalgebra of $\frg$ and $\frr$,
and fix compatible choices of positive roots $\Delta^+(\frg,\frh)\supset\Delta^+(\frr,\frh)$. The spin module $S$ can be
constructed as $\bigwedge(\frs^+)$, where $\frs^+$ is spanned by the positive root spaces in $\frg$ which are not in $\frr$.
We now set $S^+=\bigwedge^{even}(\frs^+)$ and $S^-=\bigwedge^{odd}(\frs^+)$. It follows that the spin module splits
as 
\eq 
\label{S plusminus}
S=S^+\oplus S^-,
\eeq 
and $S^\pm$ are non-isomorphic $\frr$-modules. Since $D$ has odd $C(\frs)$-part,
it interchanges $V\otimes S^+$ and $V\otimes S^-$ for any $(\frg,\frr)$-module $V$. 
This implies that $H_D(V)$ splits accordingly into
\[
H_D(V)=H_D(V)^+\oplus H_D(V)^-.
\]

Then we have the following result (\cite{hp}, 9.6.2):
\begin{proposition}
\label{six-term ic}
Assume that $\frg$ and $\frr$ have equal rank. Let $S^\pm$ and $H_D^\pm$ be defined as above. Let  
\[
0\to U\to V\to W\to 0
\]
be a short exact sequence of $(\frg,\frr)$-modules, each having an infinitesimal character. Then there is
a six-term exact sequence of the form
\[
\begin{CD}
H_D(U)^+  @>>>  H_D(V)^+  @>>> H_D(W)^+ \\
@AAA               @.           @VVV    \\
H_D(W)^-  @<<<  H_D(V)^-  @<<< H_D(U)^-
\end{CD}
\]
Moreover, the construction of this six-term sequence is natural, in the sense that the horizontal arrows are induced by the
given maps $ U\to V\to W$, and the vertical arrows, or connecting homomorphisms, are defined in a functorial way.
\end{proposition}

The above six-term sequence can be rewritten starting from the lower right corner, and then added to the above one.
The result is the six-term exact sequence 

\vspace{0.3cm}
\eq
\label{exact triangle}
\begin{CD}
H_D(U)  @>>>  H_D(V)  @>>> H_D(W) \\
@AAA               @.           @VVV    \\
H_D(W)  @<<<  H_D(V)  @<<< H_D(U)
\end{CD}
\eeq
\vspace{0.3cm}

which can in fact be collapsed onto an exact triangle  

\vspace{0.3cm}
\xymatrix{
& & & H_D(U) \ar[rr] & & H_D(V)\ar[ld] \\
& & & &  H_D(W)\ar[lu] & &  \\ 
}
\vspace{0.3cm}

This exact sequence 
can be obtained without the equal rank assumption. 
 
Moreover, we can define Dirac index of a $(\frg,\frr)$-module $V$ to be the virtual $\frr$-module
\eq
\label{def index}
I_D(V)=H_D(V)^+ - H_D(V)^-.
\eeq
Here and in the following, by a virtual $\frr$-module we mean an element of the Grothendieck group of the (semisimple) category of semisimple admissible $\frr$-modules.
Then we have the following result, see \cite{MPV}:
\begin{proposition}
\label{thm index ic}
Assume $\frg$ and $\frr$ have equal rank. Let $S^\pm$, $H_D^\pm$ and $I_D$ be defined as above.
Let $V$ be a $(\frg,\frr)$-module with infinitesimal character. Then
\[
I_D(V)=V\otimes S^+ - V\otimes S^-
\]
as virtual $\frr$-modules.
\end{proposition}

In the case when a module $V$ does not have an infinitesimal character but only a
generalized infinitesimal character, Proposition \ref{six-term ic} and Proposition \ref{thm index ic} can fail.
We show this by exhibiting an explicit example in Section \ref{sl2example}. 
Then we define the notion of higher Dirac cohomology in Section \ref{sec:hdc} (Definition \ref{defhdc}). We show that with this notion, an analogue of Proposition \ref{thm index ic} always holds (Theorem \ref{thm higher index}).
We also show that an analogue of Theorem \ref{Vogan conjecture} holds for higher Dirac cohomology (Theorem \ref{hdc vogan}). In Section
\ref{sec:triangles} we show that to any short exact sequence of $(\frg,\frr)$-modules we can attach an exact triangle as in 
(\ref{exact triangle}); the construction is however not natural. 

 Finally, in Section \ref{sec:other} we mention some other constructions of ``higher Dirac cohomology" (and homology) functors. Some of them give the same notion as the one defined in Section \ref{sec:hdc}, and some of them give different notions which might prove useful in other contexts.

\section{Motivating Examples}
\label{examples}
\subsection{An $\frs\frl(2)$ example}
\label{sl2example}

Let $\gog=\frs\frl(2,\mC)$, with standard basis $e,f,h$ satisfying the commutation relations
\begin{eqnarray}
[e,f]=h,\, [h,e]=2e,\, [h,f]=-2f.
\end{eqnarray}

For the subalgebra $\gor$ we take the one-dimensional subalgebra spanned by $h$. 
Then $(\frg,\frr)$-modules are Harish-Chandra modules for the group $SU(1,1)\cong SL(2,\bbR)$, and
they include the integral highest weight modules with respect to the Borel subalgebra $\frb$ of
$\frg$ spanned by $h$ and $e$.

We are going to construct a module $P$, which is a nontrivial extension of the two Verma
modules with trivial infinitesimal character, $V_0$ with highest $\gor$-weight 0,
and $V_{-2}$ with highest $\gor$-weight $-2$:
 
\begin{eqnarray}
\label{sesP}
0\to V_{0}\to P\to V_{-2}\to 0.
\end{eqnarray}
Recall that $V_{-2}$ is irreducible, while $V_0$ has a submodule isomorphic to 
$V_{-2}$, with quotient equal to the trivial module $\bbC_0$.

(The module we consider is denoted by $P$ because it is a projective 
object in the Bernstein-Gelfand-Gelfand category ${\fam2 O}$ of $\frs\frl(2,\mC)$ with
respect to the Borel subalgebra $\frb$. This fact is however not
important for our considerations.) 

Let us denote an $r$-module basis of 
$V_0$, resp. $V_{-2}$ by $\{v_{-2i}\}_{i\in\bbZ_+}$, resp. $\{w_{-2i-2}\}_{i\in\bbZ_+}$.
In each case, the subscript denotes the $h$-eigenvalue of the corresponding vector.
Since $P=V_0\oplus V_{-2}$ as a vector space, $v_i$ and $w_j$ form a basis
of $P$. We define the action of $e$ and $f$ on these basis elements by
\begin{eqnarray}
\label{actionP} 
& & ev_{-2k}=(-k+1)v_{-2k+2}, \nonumber \\
& & fv_{-2k}=(k+1)v_{-2k-2}, \nonumber \\
& & ew_{-2k}=(-k+1)w_{-2k+2}+\frac{1}{k}v_{-2k+2}, \nonumber \\
& & fw_{-2k}=(k+1)w_{-2k-2},
\end{eqnarray}
where $w_0$ is defined to be 0. We can describe this action by the following picture:

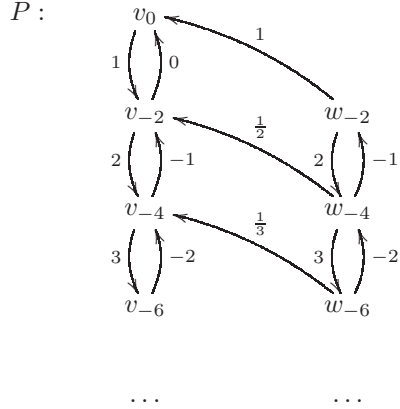
\vspace{0.3cm}
\xymatrix{
& & P:& v_0 \ar@/_/[d]_1  & & \\
& & & v_{-2} \ar@/_/[d]_2 \ar@/_/[u]_{0} & &  w_{-2} \ar@/_/[d]_2 \ar@/_/[llu]_1 \\
& & & v_{-4} \ar@/_/[d]_3 \ar@/_/[u]_{-1} & &  w_{-4} \ar@/_/[d]_3 \ar@/_/[u]_{-1} \ar@/_/[llu]_{\frac{1}{2}} \\
& & & v_{-6} \ar@/_/[u]_{-2} & & w_{-6} \ar@/_/[u]_{-2} \ar@/_/[llu]_{\frac{1}{3}}\\
& & & \dots &  & \dots
}
\vspace{0.3cm}

Here the action of $e$ is represented by upward arrows, the action of $f$ by 
downward arrows, and the numbers by the arrows represent the coefficients in the 
action computed in the basis. 

It is easy to check that we have indeed defined an $\frs\frl(2,\bbC)$-module. 
(The only thing that needs to be checked is $ef-fe=h$ on each basis vector, and that
is seen by a straightforward computation using the above formulas for the action.) 
It is moreover clear that the $v_k$ span a submodule isomorphic to $V_0$, and that the
quotient is isomorphic to $V_{-2}$, spanned by the classes of the $w_k$s. In other words, $P$
indeed fits into the short exact sequence (\ref{sesP}).

We now want to describe the action of the Dirac operator
$$
D=e\otimes f+f\otimes e
$$
on $P\otimes S$. Here $S=\bbC 1\oplus\bbC e$ is the spin module for the complexified 
Clifford algebra $C(\bbC e\oplus \bbC f)$, with action given by
\begin{eqnarray} 
& e\cdot 1=e, & e\cdot e=0;\nonumber \\
& f\cdot 1=0, & f\cdot e=-2.
\end{eqnarray}
Furthermore, the $\frr$-weight of $1$ is $-1$, while the $\frr$-weight of $e$ is $1$ (because the spin action of $\frr$ on $S$ is equal to the adjoint action shifted by $-\rho$.)

Using this and the formulas (\ref{actionP}) for the action of $\frg$ on $P$, one gets

\begin{eqnarray}
\label{D on P}
& & D(v_{-2k}\otimes 1)=(k+1)v_{-2k-2}\otimes e,\nonumber \\
& & D(v_{-2k}\otimes e)=2(k-1)v_{-2k+2}\otimes 1,\nonumber \\
& & D(w_{-2k}\otimes 1)=(k+1)w_{-2k-2}\otimes e,\nonumber \\
& & D(w_{-2k}\otimes e)=2(k-1)w_{-2k+2}\otimes 1+\frac{1}{k}v_{-2k+2}\otimes 1,
\end{eqnarray}
where for the last formula, $w_0$ should be replaced by $0$.

One can iterate these formulas and get similar formulas for the higher powers of $D$. 
It follows from (\ref{D on P}) that the generalized 0-eigenspace of $D$ is spanned by
\begin{eqnarray}
\label{gen 0 eig}
v_0\otimes 1,\, v_0\otimes e,\, v_{-2}\otimes e,\, w_{-2}\otimes e,
\end{eqnarray}
and that the action of $D$ on these vectors is given as
\begin{eqnarray}
w_{-2}\otimes e\mapsto v_0\otimes 1\mapsto v_{-2}\otimes e\mapsto 0;&\nonumber \\
v_0\otimes e\mapsto 0.&
\end{eqnarray}
This implies that the Dirac cohomology of $P$, $\Ker(D)/\Im(D)\cap\Ker(D)$, is spanned by
(the class of) $v_0\otimes e$. Namely, the other vector in $\Ker(D)$, $v_{-2}\otimes e$, is
also in $\Im(D)$.

Doing similar (but easier) explicit computations, or recalling the general result from \cite{hx}, Proposition 4.11,
one checks that the Dirac cohomology of $V_0$ is spanned by (the class of) $v_0\otimes e$, while
the Dirac cohomology of $V_{-2}$ is spanned by (the class of) $w_{-2}\otimes e$. We see that
\[
H_D^+(V_0)=H_D^+(P)=H_D^+(V_{-2})=0,
\]
while $H_D^-(V_0)$, $H_D^-(P)$ and $H_D^-(V_{-2})$ are all one-dimensional. It follows that there can
be no six-term exact sequence of the form

\vspace{0.3cm}
\xymatrix{
& & & H_D^+(V_0) \ar[r] & H_D^+(P) \ar[r] & H_D^+(V_{-2}) \ar[d] \\
& & & H_D^-(V_{-2}) \ar[u] & \ar[l] H_D^-(P)  & \ar[l] H_D^-(V_0) 
}
\vspace{0.3cm}

We now want to show that the basic property of index, Proposition \ref{thm index ic}, does not hold for $P$, i.e., that
\eq
\label{bad index}
H_D^+(P)-H_D^-(P) \neq P\otimes S^+-P\otimes S^-
\eeq
as virtual $\frr$-modules. To see this, we first note that the generalized eigenspaces for $D$ for
nonzero eigenvalues can not contribute to either side of (\ref{bad index}); see the proof of Theorem \ref{thm higher index}.
On the other hand, the contribution of the vectors (\ref{gen 0 eig}) to 
$P\otimes S^+-P\otimes S^-$ is
\[
\bbC_{-1} - \bbC_1 - \bbC_{-1} - \bbC_{-1} = - \bbC_1 - \bbC_{-1},
\]
while
\[
H_D^+(P)-H_D^-(P) =  - \bbC_1.
\]
(Here $\bbC_\lambda$ denotes the one-dimensional $\frr$-module of weight $\lambda$.)

So we conclude that (\ref{bad index}) holds, i.e., that Proposition \ref{thm index ic} fails for $P$. The reason is the fact that
the Jordan cell of length 3 present in the generalized 0-eigenspace for $D$ contributes to  $P\otimes S^+-P\otimes S^-$,
but does not contribute to $H_D(P)$. We will resolve this problem by modifying the definition of Dirac cohomology to 
what we call higher Dirac cohomology. As we will see, the Jordan cell of length 3 will contribute to the higher Dirac cohomology,
and an analogue of Proposition \ref{thm index ic} will hold for $P$ (and in general) when $H_D$ is replaced with the higher Dirac cohomology.

Note that in the above example the index is not additive with respect to short exact sequences, i.e., $I_D(P)\neq I_D(V_0)+I_D(V_{-2})$. This means 
$I_D$ is not well defined on the Grothendieck group of a category of $(\frg,\frr)$-modules which contains modules like $P$.

\subsection{Relative Dirac cohomology and branching problems for generalized Verma modules}
In fact, there is a natural reservoir of potential applications of higher 
Dirac cohomology in practice. Let us first recall some facts from \cite{hp}, Section 9.
Let $\gog=\gor\oplus\gos$ be as before, and let $\gor_1\subset\gor$ be another quadratic subalgebra
with orthogonal complement in $\frg$ equal to $\gos_1$. Then $\gog=\gor_1\oplus\gos_1$ and $\gos\subset\gos_1$,
so we can consider $C(\frs)$ as a subalgebra of $C(\frs_1)$.  
The (relative) cubic Dirac operator $D_\triangle(\gor,\gor_1)$ can be defined as the image of a certain diagonal embedding
of $U(\frr)\otimes C(\frr\cap\frs_1)$ into $U(\frg)\otimes C(\frs_1)$, so that 
$$
D(\gog,\gor_1)=D(\gog,\gor)+D_\triangle(\gor,\gor_1).
$$ 
and the summands $D_\triangle(\gor,\gor_1)$ and $D_\triangle(\gor,\gor_1)$ anticommute.
Under appropriate assumptions, some elementary linear algebra arguments imply that for a $(\frg,\frr)$-module $V$,
$H_D(V)$ can be calculated in stages: 
$$
H_{D(\gog,\gor_1)}(V)=H_{D_\triangle(\gor,\gor_1)}(H_{D(\gog,\gor)}(V))
$$ 
where $H_{D(\gog,\gor)}(V)$ is understood as an $\gor$-module. 

We would like to generalize the above considerations to the situation of an embedding of pairs of Lie algebras:
$$
(\gog',\gor')\hookrightarrow (\gog,\gor)
$$
and a $(\gog,\gor)$-module $V$. We are especially interested in the case when $\gog$ and $\gog'\subset\gog$ 
are simple Lie algebras, $\gor=\frl$ resp. $\gor'=\frl'\subset\frl$ are Levi subalgebras of 
compatible parabolic subalgebras $\gop$ resp. $\gop'\subset\gop$, and $V$ is a highest weight module for the pair
$(\frg,\frp)$. Recall that $\gop,\gop'$ are called compatible parabolic 
subalgebras provided there is a hyperbolic element in $\gop'$ responsible for the grading 
structure on the nilradical $\gon\subset\gop$.

\begin{question} 
Give the proper definition of the relative Dirac operator for compatible pairs $(\gog',\gol')\hookrightarrow (\gog,\gol)$
as above, i.e., with $\gol'\subset\gol$ the Levi subalgebras of compatible parabolic subalgebras $\frp'\subset\frp$, and
examine its role in understanding the branching problem for $(\gog,\gog')$ applied
to Bernstein-Gelfand-Gelfand parabolic category ${\fam2 O}^\gop$ (and specifically, 
to generalized Verma modules in ${\fam2 O}^\gop$). 
\end{question}

To be more specific, let us consider the couple of compatible orthogonal Lie algebras and their 
conformal parabolic subalgebras:
\begin{eqnarray}
& & \gog=\frs\fro(n+1,1,\mR)\otimes\mC,\qquad  \gop=((\frs\fro(n,\mR)\times\mR^\star)\ltimes\mR^n)\otimes\mC,
\nonumber \\
& & \gog'=\frs\fro(n,1,\mR)\otimes\mC,\qquad  \gop'=((\frs\fro(n-1,\mR)\times\mR^\star)\ltimes\mR^{n-1})\otimes\mC.
\end{eqnarray}
We denote by $\frn$ resp. $\frn'$ the nilradicals of $\frp$ resp. $\frp'$, and by $\frn_-$ resp. $\frn'_-$ the 
opposite nilradicals. Then $\gon_-,\gon_-'$ are commutative, $\gon_-'$ is of codimension 1 in $\gon_-$ and its 
one-dimensional complement is spanned by the lowest root space of $\gon_-$ (i.e., the lowest root space of $\gog$.) 

Let us consider the scalar generalized Verma-module $M^\gog_\gop(\mC_\lambda)$ for $(\gog, \gop)$,
induced from a character $\xi_\lambda: \gop\to\mC$, $\lambda\in\mC$.
The branching problem for the couple $(\gog,\gop), (\gog',\gop')$ and this class of modules
was recently studied by Kobayashi, \O rsted, Somberg and Sou\v cek \cite{koss}, \cite{kossII}. The approach in this
work was analytic, using a suitable Fourier transform to turn 
the combinatoral problem for finding the singular vectors into the question about solutions 
of the system of partial differential equations of hypergeometric type (the ``F-method").
For generic values of $\lambda\in\mC$, there is a $\gog'$-module isomorphism
\begin{eqnarray}\label{decompverma}
M^\gog_\gop(\lambda)|_{(\gog',\gop')}\cong \oplus_{j=0}^\infty M^{\gog'}_{\gop'}(\lambda-j),
\end{eqnarray}
where the $\gog'$-modules on the right hand side are induced from the $1$-dimensional $\frp'$-modules $\mC_{\lambda-j}$,
and their singular vectors are given by Gegenbauer polynomials.

It follows from \cite{hx} that the $(\gog,\gol)$-Dirac cohomology of the left hand side is
$$ 
H_{D(\gog,\gol)}(M^\gog_\gop(\lambda))= \mC_\lambda\otimes\mC_{\rho(\gon_-)},
$$ 
while the $(\gog',\gol')$-Dirac cohomology of the right hand side is 
$$
H_{D(\gog',\gol')}(\oplus_{j=0}^\infty M^{\gog'}_{\gop'}(\lambda-j))= 
\oplus_{j=0}^\infty\mC_{\lambda-j}\otimes\mC_{\rho(\gon_-')}. 
$$  
Here $\rho(\gon_-)$ resp. $\rho(\gon_-')$ denote the half sums 
of roots in $\gon_-$ resp. $\gon_-'$.

We can define a copy of the Lie algebra $\frs\frl(2,\bbC)$ corresponding to the one-dimensional
complement of $\frn_-'$ in $\frn_-$, which acts on the singular vectors
of the modules $M^{\gog'}_{\gop'}(\lambda-j)$ in the above decomposition. Denoting by $D_\Delta$
the corresponding Dirac operator, we see that
\begin{eqnarray}\label{stages}
H_{D(\gog,\gol)}(M^\gog_\gop(\lambda))=
H_{D_\triangle}(H_{D(\gog',\gol')}(M^\gog_\gop(\lambda)|_{\gog'})).
\end{eqnarray}

However, for the non-generic values of inducing parameter $\lambda\in\mC$,
on the right hand side of (\ref{decompverma}) there appear non-trivial 
extensions in ${\fam2 O}^{\gop'}$ of generalized Verma modules.
In these cases, (\ref{stages}) is no longer true, and to obtain an analogous statement
one has to replace the ordinary Dirac cohomology by the higher Dirac cohomology which we define
in this article.

These questions together with many other examples will be
discussed in \cite{ps2}. 


\section{Higher Dirac cohomology}
\label{sec:hdc}

As we have seen in the previous section, the usual Dirac cohomology functor $V\mapsto H_D(V)$ does not
behave well on finite length modules $V$ which do not have infinitesimal character. Our goal in this section
is to define a new version of Dirac cohomology, $V\mapsto H(V)$, with better behavior. 
In case $V$ has infinitesimal character, we want $H(V)$ to be the same as $H_D(V)$.
Moreover, in the equal rank case, $S=S^+\oplus S^-$ should induce $H(V)=H(V)^+\oplus H(V)^-$ for any $V$, in such a way that 
for any $V$ of finite length, $V\otimes S^+-V\otimes S^-=H(V)^+-H(V)^-$. 

We would also like to have an analogue of Vogan's conjecture (Theorem \ref{Vogan conjecture}) for $H(V)$, where $V$ is a module with generalized infinitesimal character. 

Finally, if $0\to U\to V\to W\to 0$ is a short exact sequence of $(\frg,\frr)$-modules, there should be an exact triangle
\bigskip

\xymatrix{
& & & H(U) \ar[rr] & & H(V)\ar[ld] \\
& & & &  H(W)\ar[lu] & &  \\ 
}
\bigskip

There are quite a few possible constructions that satisfy many of the above properties; we list some of them in Section \ref{sec:other}. 
A definition that comes close to satisfying all of the above properties is as follows.

\begin{definition}
\label{defhdc}
For any integer $k\geq 0$, define 
\eq
\label{hdc1}
H^k(V) = \im D^{2k}\cap \ker D \big/ \im D^{2k+1}\cap\ker D.
\eeq
We further define 
\eq 
\label{hdc2}
H(V)=\bigoplus_{k\in\bbZ_+} H^k(V).
\eeq 
We call $H^k(V)$ and $H(V)$ the higher Dirac cohomology of $V$.
\end{definition}

Note that $H^0(V)=\ker D \big/ \im D\cap\ker D$ is the old notion $H_D(V)$.

Let $V$ be a $(\frg,\frr)$-module and let $(V\otimes S)_{[0]}$ denote the generalized 0-eigenspace for $D$ acting on $V\otimes S$.
Then $(V\otimes S)_{[0]}$ is a finite-dimensional space, and it can be decomposed into a direct sum of Jordan blocks for $D$ of various sizes. 
By a Jordan block of size $k$ we mean a space
\eq
\label{block}
V_1\oplus V_2\oplus \dots \oplus V_k,
\eeq
where the $V_i$ are copies of a single $\frr$-type, and 
\eq
\label{block2} 
\begin{CD}
D(V_1)=0; \qquad D: V_i @>\cong>> V_{i-1},\quad i=2,\dots, k, 
\end{CD}
\eeq
while $V_k$ is not in the image of $D$.
Note that since $D$ is $\frr$-invariant, it can only map an $\frr$-module onto
an isomorphic $\frr$-module, so we can have a block as above only with copies of the same $\frr$-type.

\begin{theorem}
\label{hdc block}
Let $V$ be a $(\frg,\frr)$-module, and let $k\geq 0$ be an integer. Choose any decomposition of $(V\otimes S)_{[0]}$ into Jordan blocks, and let
$E_1,\dots,E_r$ be the bottom $\Kt$-types of all blocks of size $2k+1$. Then 
\[
H^k(V)\cong E_1\oplus\dots\oplus E_r.
\]
In particular, $H(V)$ is isomorphic to the direct sum of all bottom $\Kt$-types of odd size Jordan blocks.
\end{theorem}

\begin{proof}
Note that $\ker D$ is exactly the sum of the bottom $\Kt$-types in all blocks in the decomposition of $(V\otimes S)_{[0]}$. Let $E$ be one such bottom
$\tilde K$-type, in a Jordan block of length $m$. 

If $m<2k+1$, then $E$ is not contained in $\im D^{2k}$ and hence $E$ does not contribute to $H^k(V)$. 
If $m>2k+1$, then $E$ is contained in $\im D^{2k+1}$ and hence $E$ does not contribute to $H^k(V)$. 
Finally, if $m=2k+1$, then $E$ is contained in $\im D^{2k}$ and not contained in $\im D^{2k+1}$, so $E$ contributes to $H^k(V)$.
\end{proof}

\begin{corollary} If $V$ is a $(\frg,\frr)$-module with infinitesimal character, then $H(V)=H^0(V)=H_D(V)$.
\end{corollary}

\begin{proof} Since $V$ has infinitesimal character, $D^2$ is a semisimple operator on $V\otimes S$, so Jordan blocks in $(V\otimes S)_{[0]}$ 
can only have sizes 1 and 2. It follows that $H^k(V)=0$ for $k>0$, and this implies the claim.
\end{proof}

We are now going to prove a generalization of Vogan's conjecture, Theorem \ref{Vogan conjecture}. 
As there, we take a Cartan subalgebra $\frh$ of $\frg$ containing
a Cartan subalgebra $\frt$ of $\frr$, and we view $\frt^*$ as a subspace of $\frh^*$.

\begin{theorem} 
\label{hdc vogan}
Let $V$ be a $(\frg,\frr)$-module with generalized infinitesimal character corresponding to $\Lambda\in\frh^*$. Let $E_\gamma$ be
a $\tilde K$-type in $H(V)$. Then $\Lambda$ is $W_\frg$-conjugate to $\gamma+\rho_\frk$. 
\end{theorem}

\begin{proof}

Recall from \cite{HP1} or \cite{hp} that the proof for $H_D(V)$, in case when $V$ has infinitesimal character, is based on writing 
$z\in Z(\frg)$ as
\[
z\otimes 1 = \zeta(z) + Da+aD,
\]
with $\zeta(z)\in Z(\frk_\Delta)$ and $a\in (U(\frg)\otimes C(\frp))^K$.
Since $Da+aD$ acts as 0 on $H_D(V)$, $z\otimes 1$ acts in the same way as $\zeta(z)$ and the claim follows, since $\zeta:Z(\frg)\to Z(\frk_\Delta)$ can be written explicitly in terms of Harish-Chandra isomorphisms.

In the present situation, we only need to show that $Da+aD$ acts as 0 on $H^k(V)$ for each $k$. It is clear that
$aD$ acts as 0, so it remains to show that $Da$ acts as 0.

Let $x\in\im D^{2k}\cap\ker D$, so $x=D^{2k}y$ and $Dx=0$. 
Since $D^2$ is central in $(U(\frg)\otimes C(\frp))^K$,
\[
Dax=DaD^{2k}y=D^{2k+1}ay\in \im D^{2k+1}.
\]
Moreover, $D(Dax)=aD^2x=0$, so $Dax\in \im D^{2k+1}\cap\ker D$. So $Da$ sends the class of $x$ to 0.
\end{proof}

For the rest of this section we assume that $\frg$ and $\frr$ have equal rank, and fix a $(\frg,\frr)$-module $V$. 
We decompose the spin module $S$ as in (\ref{S plusminus}): $S=S^+\oplus S^-$. 
Since $D$ has odd $C(\frs)$-part,
it interchanges $V\otimes S^+$ and $V\otimes S^-$ for any $(\frg,\frr)$-module $V$. 
Thus all odd powers of $D$ interchange $V\otimes S^+$ and $V\otimes S^-$,
while the even powers preserve them. This implies that each $H^k(V)$ splits into even and odd parts, and consequently
so does $H(V)$:
\[
H(V)=H(V)^+\oplus H(V)^-.
\]
We define the higher Dirac index of $V$ to be the virtual $\frr$-module
\eq
\label{def higher index}
I(V)=H(V)^+ - H(V)^-.
\eeq
Then we have the following result:
\begin{theorem}
\label{thm higher index}
Assume $\frg$ and $\frr$ have equal rank. Let $V$ be a $(\frg,\frr)$-module. Then
\[
I(V)=V\otimes S^+ - V\otimes S^-
\]
as virtual $\frr$-modules.
\end{theorem}

\begin{proof}
By our assumptions on $V$, $V$ splits into a direct sum of components with generalized infinitesimal characters. 
It then follows from (\ref{D squared}) that $V\otimes S$ splits into a direct sum of generalized eigenspaces $(V\otimes S)_\lambda$ for $D^2$.
This decomposition is compatible with the decomposition $V\otimes S=V\otimes S^+ \oplus V\otimes S^-$.

It is clear that $D$ preserves each $(V\otimes S)_\lambda$. Moreover, $D$ is invertible on $(V\otimes S)_\lambda$ if $\lambda\neq 0$. 
Indeed, $D$ is injective on $(V\otimes S)_\lambda$ since $Dv=0$ implies $D^2v=0$ and for $v\in (V\otimes S)_\lambda$ this is impossible unless
$v=0$. Since $(V\otimes S)_\lambda$ is finite-dimensional, $D$ must also be surjective on it.

Now since
\[
\begin{CD}
D:(V\otimes S)_\lambda^+ @>\cong>> (V\otimes S)_\lambda^-,\qquad \lambda\neq 0,
\end{CD}
\]
we see that in $V\otimes S^+-V\otimes S^-$ all $(V\otimes S)_\lambda$ for $\lambda\neq 0$ cancel out.
Since $H(V)$ is constructed only from $(V\otimes S)_0$, we see that it is enough to prove
\[
I(V)=(V\otimes S)_0^+ - (V\otimes S)_0^-.
\]
We now decompose $(V\otimes S)_0$ into Jordan blocks. Each such block is of the form $V_1\oplus\dots\oplus V_k$ as in 
(\ref{block}) and (\ref{block2}), and we can moreover assume that $V_1$ is either in $(V\otimes S)_0^+$ (even) or in $(V\otimes S)_0^-$ (odd).
Since $D$ interchanges $(V\otimes S)_0^+$ and $(V\otimes S)_0^-$, it follows that $V_1,V_2,V_3,\dots$ are of alternating parity.
We now conclude that the total contribution of this Jordan block to $V\otimes S^+-V\otimes S^-$ is zero if $k$ is even, and $\pm V_1$ if $k$ is odd.
(The sign $\pm$ depends on the parity of $V_1$.) By  Theorem \ref{hdc block}, the contribution to $I(V)$ is exactly the same, and the result follows.
\end{proof}

\begin{corollary}
\label{index Groth}
The higher Dirac index is additive with respect to short exact sequences of $(\frg,\frr)$-modules, and hence
makes sense on the level of virtual $(\frg,\frr)$-modules.
\end{corollary}
\begin{proof} This follows from Theorem \ref{thm higher index} and its proof. Namely, the statement of the corollary is clear for the functor
$V\mapsto V\otimes S^+-V\otimes S^-$, since a short exact sequence of $(\frg,\frr)$-modules splits in the category of $\frr$-modules, and since
we have seen that in $V\mapsto V\otimes S^+-V\otimes S^-$ all terms but finitely many of them cancel out.
\end{proof}

\section{Exact triangles}
\label{sec:triangles}

\begin{theorem}
\label{thm:triangles}
Let $0\to U\to V\to W\to 0$ be a short exact sequence of $(\frg,\frr)$-modules. Then there is an exact triangle
\bigskip

\xymatrix{
& & & H(U) \ar[rr] & & H(V)\ar[ld] \\
& & & &  H(W)\ar[lu] & &  \\ 
}
\bigskip
\end{theorem}

One would like the maps $H(U)\to H(V)\to H(W)$ to be induced by the maps $U\to V\to W$, and to have a naturally
defined connecting homomorphism $H(W)\to H(U)$. 
We were however unable to obtain such a natural construction (indeed, we believe that it is not possible).
We instead use the Jordan block decomposition to obtain the above exact triangle in a noncanonical fashion.

\begin{proof}
We first make compatible Jordan block decompositions of $(U\otimes S)_{[0]}$, $(V\otimes S)_{[0]}$ and $(W\otimes S)_{[0]}$. We start by picking a decomposition
for $(W\otimes S)_{[0]}$ into Jordan blocks $W_1,\dots,W_n$. Let $\bar W_1,\dots,\bar W_n$ be the top $\tilde K$-types of the blocks $W_1,\dots,W_n$. 
We choose preimages $\bar V_1,\dots,\bar V_n$ of $\bar W_1,\dots,\bar W_n$ in $(V\otimes S)_{[0]}$. These 
are necessarily also top $\tilde K$-types of some blocks $V_1,\dots,V_n$. Clearly, for each $i$, the kernel of $V_i\to W_i$ is etiher zero, or defines a Jordan block $U_i$ in $(U\otimes S)_{[0]}$.

Now we choose a direct complement of the sum of the blocks $V_1,\dots,V_n$, which is contained in the kernel of 
$V\to W$, and choose any
Jordan block decomposition for this complement. We have now obtained a decomposition for $(V\otimes S)_{[0]}$ and also for $(U\otimes S)_{[0]}$.

Our short exact sequence (tensored with $S$) now breaks up into a direct sum of short exact sequences
\[
0\to U_j\to V_j\to W_j\to 0,
\]
where each of $U_j$, $V_j$ and $W_j$ is either a single Jordan block or zero. It is therefore enough to obtain an exact triangle for each $j$.

Since the lengths of $U_j$ and $W_j$ add up to the length of $V_j$, there are two possibilities:

\begin{enumerate} 
\item $U_j$, $V_j$ and $W_j$ are all even. In that case $H(U_j)=H(V_j)=H(W_j)=0$, and hence the exact triangle trivially exists.
\item One of $U_j$, $V_j$ and $W_j$ is even, while the other two are odd. The cohomology of the two odd blocks is the same single
$\tilde K$-type, while the cohomology of the even block is 0. The exact triangle is one of the following.
\end{enumerate} 

\xymatrix{
& & & H(U_j) \ar[rr]^1 & & H(V_j)\ar[ld]^0 \\
& & & &  H(W_j)\ar[lu]^0 & &  \\ 
}
\bigskip

\xymatrix{
& & & H(U_j) \ar[rr]^0 & & H(V_j)\ar[ld]^1 \\
& & & &  H(W_j)\ar[lu]^0 & &  \\ 
}
\bigskip

\xymatrix{
& & & H(U_j) \ar[rr]^0 & & H(V_j)\ar[ld]^0 \\
& & & &  H(W_j)\ar[lu]^1 & &  \\ 
}
\end{proof}

Since $U$, $V$ and $W$ are all of finite length, the spaces $H(U)$, $H(V)$ and $H(W)$ are finite-dimensional. Thus we have:

\begin{corollary} The existence of the above exact triangle is equivalent to the following two conditions:
\begin{enumerate}
\item $\dim H(U) + \dim H(V)+\dim H(W)$ is even;
\item $\dim H(U)$, $\dim H(V)$ and $\dim H(W)$ are sides of a (possibly degenerate) triangle, i.e., 
the sum of any two of them is greater than or equal to the third.
\end{enumerate}
\end{corollary}

\begin{proof} 
This follows from the fact that existence of an exact triangle with vertices 
$H_1$, $H_2$ and $H_3$ is equivalent to the condition that $H_i$ can be written as
$H_i=A_i\oplus B_i$, so that there are isomorphisms $B_1\to A_2$, $B_2\to A_3$ and $B_3\to A_1$.

Writing the dimensions using lower case letters, it follows that $b_1=a_2$, $b_2=a_3$ and $b_3=a_1$.
So the above condition is equivalent to 
\[
h_1=a_1+a_2,\qquad h_2=a_2+a_3,\qquad h_3=a_3+a_1.
\]

The solution to this system of equations is
\[
a_1=\frac{h_1-h_2+h_3}{2},\quad a_2=\frac{h_1+h_2-h_3}{2},\quad a_3=\frac{-h_1+h_2+h_3}{2}.
\]

Since the $a_i$ have to be nonnegative integers, we are led to the above conditions on $h_i$. 
\end{proof}

\section{Other related constructions}
\label{sec:other}

Definition \ref{defhdc} of $H^k(V)$ can be thought of as extracting the bottom $\Kt$-type
from each Jordan block of size $2k+1$. One can similarly extract the top $\Kt$-type: define 
\eq 
\label{eq:top}
H^k_{\top}(V)=\ker D^{2k+1}\big/ (\im D\cap\ker D^{2k+1}+\ker D^{2k}).
\eeq

\begin{proposition}
\label{top bottom}
$H^k_{\top}(V)$ (defined by (\ref{eq:top})) is naturally isomorphic to $H^k(V)$ (defined by (\ref{hdc1})).
\end{proposition}
\begin{proof} The isomorphism from $H^k_{\top}(V)$ to $H^k(V)$ is induced by $D^{2k}$. 

Indeed, it is clear that $D^{2k}$ maps
$\ker D^{2k+1}$ onto $\im D^{2k}\cap \ker D$. Denote by $\varphi$ the composition of this map with the projection 
$\im D^{2k}\cap \ker D\to H^k(V)$. It suffices to show that $\ker\varphi= \im D\cap\ker D^{2k+1}+\ker D^{2k}$. It is
clear that $\im D\cap\ker D^{2k+1}$ and $\ker D^{2k}$ are contained in $\ker\varphi$, hence so is their sum. Conversely,
assume $x\in \ker D^{2k+1}$ is such that $\varphi(x)=0$. Then $D^{2k} x \in  \im D^{2k+1}\cap \ker D$. So there is $y$ such
that $D^{2k}x=D^{2k+1}y$, i.e., $D^{2k}(x-Dy)=0$. Denoting $x-Dy$ by $z$, we see that $x=Dy+z$, with $Dy\in \im D\cap\ker D^{2k+1}$
and $z\in \ker D^{2k}$. This finishes the proof.
\end{proof}

It is also possible to apply lower powers of $D$ to $H^k_{\top}$, to extract other $\Kt$-types from the Jordan blocks of size $2k+1$.
However, in order to keep Theorem \ref{thm higher index} valid, one should only extract cells of correct parity, i.e.,
apply only even powers of $D$. Explicitly, the formula for extracting the $(2k+1-2i)$-th $\Kt$-type is
\[
H^k_i(V)=\im D^{2i}\cap\ker D^{2k+1-2i}\big/ (\im D^{2i+1}\cap\ker D^{2k+1-2i}+\im D^{2i}\cap\ker D^{2k-2i}).
\]
We leave it to the reader to prove an analogue of Proposition \ref{top bottom} for $H^k_i$.

There are other kinds of ``higher cohomology" one can attach to $D$, discussed in \cite{DV}. In the
following we briefly describe them and show how they are related to our definition.

Let $V$ be a $(\gog,\gor)$-module. Let us fix a number $N\in\mN$ such that $D^N=0$ on 
the generalized 0-eigenspace $(V\otimes S)_{[0]}$ for $D$ acting on $V\otimes S$. The number $N$ can be chosen
as the size of the largest Jordan block for $D$ on $(V\otimes S)_{[0]}$, or $N$ can be any larger number.

Now $D$ is an $N$-differential on $(V\otimes S)_{[0]}$ in the sense of \cite{DV}, and we can consider the following
cohomology spaces:
\eq
\label{hdcDV}
H^i_D(V)=H^i_D((V\otimes S)_{[0]})=\Ker(D^i)/\Im(D^{N-i}),\qquad i=1,\dots ,N-1.
\eeq

The definition depends on the choice of $N$, and in general it is quite different from our $H^k(V)$. We will see below that
there is a relationship between the two notions in general; for now, we just remark that if $V$ has infinitesimal character,
then $N$ can be taken to be 2, and then
\[
H^1_D(V)=H_D(V)=H(V).
\]
The main advantage of the functors (\ref{hdcDV}) is the existence of certain naturally defined  
six-term exact sequences attached to a short exact sequence 

\xymatrix{
& & & 0 \ar[r] & U \ar[r]^{i} & V \ar[r]^{p} & W \ar[r] & 0
}
\smallskip
\noindent of $(\frg,\frr)$-modules as in \cite{DV}, Lemma 2; see also \cite{KW}.

Namely, choosing $N$ as above, one gets that for any $i$ between 1 and $N-1$ there is a six term exact sequence
\smallskip

\xymatrix{
 & & H_D^i(U) \ar[r]^{i_\star} & H_D^i(V) \ar[r]^{p_\star} & H_D^i(W) \ar[d]^\partial \\
 & & H_D^{N-i}(W) \ar[u]^{\partial} & \ar[l]_{p_\star} H_D^{N-i}(V)  & \ar[l]_{i_\star} H_D^{N-i}(U) 
}
\smallskip

Here $i_*$ and $p_*$ are induced by $i$ and $p$ in the obvious way, while the connecting homomorphism 
\begin{eqnarray} 
\partial : H^i_D(W)\to H^{N-i}_D(U)
\end{eqnarray}
is defined in the following way. Let $w\in W$ be a representative of an element of $H^i_D(W)$, i.e. $D^iw=0$. By 
surjectivity of $p$ there is $v\in V$ such that $p(v)=w$. Define $z=D^iv$, then 
$D^{N-i}z=0$. We have $p(z)=p(D^{i}v)=D^{i}p(v)=D^{i}w=0$, hence there is
$u$ such that $z=i(u)$. Because $D^{N-i}u=0$, we set 
$\partial(w)=[u]\in \Ker(D^{N-i})/\Im(D^{i})$.

It is now easy to check that the class $[u]$ is independent of the choice of $v$, that the map $\partial$ is well 
defined on classes, and that the above six-term sequence is exact. 

To describe the relationship of the functors (\ref{hdcDV}) to our $H(V)$, we first consider the case when 
$(V\otimes S)_{[0]}$ is a single Jordan block of length $k$, as in (\ref{block}) and (\ref{block2}). 
It is clear that $D^N=0$ on $(V\otimes S)_{[0]}$ if and only if $N\geq k$; we fix such an $N$. 

Since there is a compact Lie group with complexified Lie algebra $\frr$, we can
construct an $\frr$-invariant inner product on $V$, and we can moreover assume that all the $V_i$ contained in the 
block $(V\otimes S)_{[0]}$ are orthogonal to each other with respect to this inner product. Using this inner product, 
we identify a quotient of the form $\ker D^i/\im D^{N-i}$ with the orthogonal complement of $\im D^{N-i}$ in $\ker D^i$.

\begin{proposition}
\label{prop block}
With the above notation, $V_j$ contributes to $H_D^i(V)$ if and only if
\[
j\leq i < j+N-k.
\]
\end{proposition}
\begin{proof} $V_j$ will contribute to $H_D^i(V)$ precisely when $D^i(V_j)=0$ and $V_j$ is not contained in $\im D^{N-i}$. Clearly,
$D^i(V_j)=0$ if and only if $j\leq i$. On the other hand, $V_j$ is contained in $\im D^{N-i}$ if and only if $j+N-i\leq k$. Namely, if this 
inequality holds, then 
\[
V_j=D^{N-i}(V_{j+N-i}),
\]
and that is the only possibility to obtain $V_j$ in $\im D^{N-i}$. So $V_j$ is not contained in $\im D^{N-i}$ if and only if $j+N-i> k$ and the 
statement follows. 
\end{proof}

\begin{example}{\rm Still with the same notation as above, assume that $N=k$. Then by Proposition \ref{prop block}, $H_D^i(V)=0$ for all $i$. Namely,
the inequality $j\leq i<j+N-k=j$ is impossible.

Assume now that $N=k+1$. Then $j\leq i< j+N-k=j+1$ is satisfied precisely for $i=j$. So it follows that in this case
\[
H_D^i(V)=V_i,\qquad 1\leq i\leq N-1=k.
\]
For $N=k+2$, we see that each $V_j$ contributes to $H_D^j(V)$ and $H_D^{j+1}(V)$. It follows that $H_D^1(V)=V_1$, $H_D^{N-1}(V)=V_k$, and
$H_D^i(V)=V_i\oplus V_{i-1}$ for $2\leq i\leq k=N-2$. One can similarly write down what $H_D^i(V)$ is in general. 
}
\end{example}

\begin{proposition}
\label{prop stable}
Let $V$ be a $(\frg,\frr)$-module, and let $2N$ be a positive even integer such that $D^{2N}=0$ on $(V\otimes S)_{[0]}$.
Then there is an equality of virtual $\frr$-modules
\[
H(V) = \sum_{i=1}^{2N-1} (-1)^{i-1} H_D^i(V).
\]
\end{proposition}

\begin{proof} We can decompose $(V\otimes S)_{[0]}$ into Jordan blocks, and prove the required equality for each block separately.
So let us assume that $(V\otimes S)_{[0]}$ is a single block, of size $k$.

By Proposition \ref{prop block}, each $V_j$ contributes to $H^i_D(V)$ precisely for 
\[
i=j,j+1,\dots, j+2N-k-1. 
\]
Since these are neighbouring degrees, 
the consecutive contributions cancel in the alternating sum. 

If $k$ is even, then $2N-k-1$ is odd. So $V_j$ appears an even number of times, and it cancels completely in the alternating sum. This is true for every $j$, so $\sum_{i=1}^{2N-1} (-1)^{i-1} H_D^i(V)=0$. 

If $k$ is odd, then $V_j$ appears an odd number of times, so it contributes to the alternating sum once, with the coefficient $(-1)^{j-1}$. This implies that
\[
\sum_{i=1}^{2N-1} (-1)^{i-1} H_D^i(V)= \sum_{j=1}^{k} (-1)^{j-1} V_j,
\]
and this is equal to $V_1$ since $k$ is odd, and all $V_j$ are isomorphic to $V_1$. Thus we see that $\sum_{i=1}^{2N-1} (-1)^{i-1} H_D^i(V)$ is the same as $H(V)$ as described by Theorem \ref{hdc block}. 
\end{proof}

\begin{remark}{\rm
Note that in the statement of Proposition \ref{prop stable}, the individual summands $H_D^i(V)$ depend on the choice of $N$. Their alternating sum 
is however equal to $H(V)$, so it is independent of the choice of $N$, as long as $N$ is sufficiently large,
i.e., $D^{2N}=0$ on $(V\otimes S)_{[0]}$. This fact can be considered as a stability property of the $H_D^i(V)$.
}
\end{remark}

\begin{remark}{\rm
Using the fact that in the Grothendieck group quotients can be written as differences, one can rewrite the alternating sum of 
Proposition \ref{prop stable} as 
\[
\sum_{i\geq 0} (\Ker D^{2i+1}/\Ker D^{2i})\big/ D(\ker D^{2i+2}/\Im D^{2i+1}),
\]
where $D:\ker D^{2i+2}/\Im D^{2i+1}\to \Ker D^{2i+1}/\Ker D^{2i}$ is an embedding induced by $D$. Another similar expression giving the same result is
\[
\sum_{i\geq 0} \Ker(D:\, \Im D^{2i}/\Im D^{2i+1}\longrightarrow \Im D^{2i+1}/\Im D^{2i+2}).
\]
}
\end{remark}

\begin{remark}{\rm
There are still more candidates for higher Dirac cohomology or homology functors.
They do not have the properties we wanted in this paper, but they might show to be useful for some other purposes.

Let $V$ be a $(\gog,\gor)$-module as before. Let $N$ be sufficiently large so that $D^N=0$ on $(V\otimes S)_{[0]}$.
Consider the following filtration of $(V\otimes S)_{[0]}$:
\[
0\subset \Ker(D)\subset \Ker(D^2)\subset\dots\subset \Ker(D^N)=(V\otimes S)_{[0]}.
\]
Notice that for any $i$ between $0$ and $N-1$, $D(\Ker(D^{i+1}))\subset \Ker(D^{i})$. 
Thus we can define cohomology of $D$ as
\[
\tilde H^i(V) = \Coker(D: \Ker(D^{i+1})\to \Ker(D^{i})).
\]
Analogously, we can define homology of $D$ as
\[
\tilde H_i(V) = \Ker(D: \Coker(D^{i})\longrightarrow \Coker(D^{i+1})).
\]
For example, one can easily check that
\[
\tilde H^1(V) =\tilde H_1(V)=\Ker(D)/(\Ker(D)\cap \Im(D)),
\]
and that for any $k\leq N$ such that $D^k=0$ on $(V\otimes S)_{[0]}$,
\[
\tilde H^{k}(V) =\Coker(D), \qquad \tilde H_{k}(V) =\Ker(D).
\]
}
\end{remark}

\end{document}